\documentclass[reqno]{amsart} 

\usepackage[latin1]{inputenc}

\usepackage{amsmath,amssymb,bbm}
\usepackage{enumitem}
\usepackage{color}

\definecolor{darkgreen}{rgb}{0.00,0.50,0.10}
\definecolor{lightgreen}{rgb}{0.20,0.70,0.30}
\usepackage[colorlinks,bookmarks,linkcolor=darkgreen,citecolor=lightgreen]{hyperref}

\newtheorem{theorem}{Theorem}
\newtheorem{lemma}[theorem] {Lemma}   
   
\newtheorem{conjecture}[theorem] {Conjecture}

\newtheorem{definition}[theorem] {Definition} 
\theoremstyle{remark}

\newcommand{\An}{\textnormal{And}}
\newcommand{\floor}[1]{\ensuremath{\left \lfloor #1 \right \rfloor}}

\title[Bipartification of Andr\'{a}sfai graphs]{The Erd\H{o}s bipartification conjecture is true in the special case of Andr\'{a}sfai graphs}

\author{Peter Heinig}
\address{Zentrum Mathematik, Lehr- und Forschungseinheit M9 f\"{u}r Angewandte Geometrie und Diskrete Mathematik, 
Technische Universit\"at M\"unchen, Boltzmannstra\ss{}e~3, D-85747 Garching bei M\"unchen, Germany} 
\email{heinig@ma.tum.de}

\thanks{The author was supported by TopMath, an Elite Graduate Program of the ENB}

\begin{document}

\maketitle

\begin{abstract} Let the Andr\'{a}sfai graph $\An_k$ be defined as the 
graph with vertex set $\{v_0,v_1,\dotsc, v_{3k-2}\}$ and two vertices $v_i$ 
and $v_j$ being adjacent iff $|i-j| \equiv 1\mod 3$. The graphs $\An_k$ are maximal triangle-free and 
play a role in characterizing triangle-free graphs with large minimum degree as homomorphic preimages.
A minimal bipartification of a graph $G$ is defined as a set of edges $F\subset E(G)$ having 
the property that the graph $(V(G), E(G)\backslash F)$ is bipartite and for every $e \in F$ the 
graph $(V(G), E(G)\backslash (F\backslash e))$ is not bipartite. In this note it is shown that there 
is a minimal bipartification $F_k$ of $\An_k$ which consists of exactly 
$\floor{\frac{k^2}{4}}$ edges. This equals $\floor{\frac{1}{36}\bigl( |\An_k|+1\bigr)^2}$, 
where $|\cdot|$ denotes the number of vertices of a graph. For all $k$ this is consistent with 
a conjecture of Paul Erd\H{o}s that every triangle-free graph $G$ can be made bipartite by deleting 
at most $\frac{1}{25} |G|^2$ edges. 

Bipartifications like $F_k$ may be useful for proving that arbitrary homomorphic preimages 
of an Andr\'{a}sfai graph can be made bipartite by deleting at most $\frac{1}{25} |G|^2$ edges.
\end{abstract}

\section{Introduction}
All notation is standard and follows ~\cite{Diestel}. In particular, if $G$ is a graph then $|G|$ 
denotes the number of its vertices. A \emph{minimal bipartification} of a graph $G$ is defined as 
a set of edges $F\subseteq E(G)$ having the property that the graph $(V(G), E(G)\backslash F)$ is 
bipartite and for every $e \in F$ the graph $(V(G), E(G)\backslash (F\backslash e))$ is not bipartite. 
A \emph{homomorphic preimage} of a graph $G$ is a preimage of $G$ under some graph homomorphism.
The present short note is concerned with the following special class of graphs.
\begin{definition}[Andr\'{a}sfai graphs] For every integer $k\geq 2$ the 
graph $\An_k$ is defined as the graph with vertex set $\{v_0,v_1,\dotsc, v_{3k-2}\}$ 
and two vertices $v_i$ and $v_j$ being adjacent iff $|i-j| \equiv 1\mod 3$. 
\end{definition}
By Lemma 6.10.1 in ~\cite{GodRoy}, every graph $\An_k$ is a triangle-free graph of diameter two, 
which is the same as saying that it is \emph{maximal} triangle-free. 

In the proof below, the following lemma will be used for the inductive step.
\begin{lemma}[Inductive construction of Andr\'{a}sfai graphs]\label{lem:AndInd}
Deleting from $\An_k$ the path $v_{3k-4}v_{3k-3}v_{3k-2}$ leaves the graph $\An_{k-1}$.
\end{lemma}
\begin{proof} This is stated above Lemma 6.11.2 in ~\cite{GodRoy} and easy to see from the definition of $\An_k$.
\end{proof}

\section{Main result}
The following theorem exhibits a minimal bipartification for the graphs $\An_k$.
\begin{theorem}\label{thm:bipForAnd}
For every integer $k\geq 2$ the set of edges $F_k := U_k^{(1)} \cup U_k^{(2)}$, where
\begin{align}
U_k^{(1)} & := \bigcup_{i=0}^{\floor{\frac{k}{2}}-1}\bigcup_{j=i}^{\floor{\frac{k}{2}}-1} \bigl \{ \{ v_{(3k-4)-3i}, v_{(3k-5)-3j} \}\bigr \},\label{Uk1}\\
U_k^{(2)} & := \bigcup_{i=0}^{\floor{\frac{k-1}{2}}-1}\bigcup_{j=0}^{\floor{\frac{k-1}{2}}-1-i} \bigl \{ \{ v_{3i}, v_{(3i+1)+3j} \}\bigr \}\label{Uk2},
\end{align}
is a minimal bipartification of $\An_k$ in the stronger sense that omitting an element from it creates a $5$-cycle. 
Moreover, the bipartite graph $\An_k-F_k$ admits the bipartition $A_k\cup B_k$ where
\begin{gather}
A_k := \{ v_{3i}:\; i=0,\dotsc, k-1\}\cup \{v_{3i+1}:\; i=0,\dotsc, \floor{(k-1)/2}-1 \},\\
B_k := \{ v_{3i+2}:\; i=0,\dotsc, k-2\}\cup \{v_{3i+1}:\; \floor{(k-1)/2},\dotsc, k-1 \}.
\end{gather}
Moreover, the set $F_k$ consists of exactly $\floor{\frac{k^2}{4}}=\floor{\frac{1}{36}\bigl( |\An_k|+1\bigr)^2}$ edges.
\end{theorem}

\begin{proof} This will be proved by induction on $k$. For $k=2$, 
the graph $\An_k$ is the $5$-cycle $v_0v_1v_2v_3v_4v_5v_0$ and the lemma correctly 
states that the single edge $\{v_1,v_2\}$ is a minimal bipartification in the stronger 
sense stated above and that $v_0v_1v_2v_3v_4v_5v_0-\{v_1,v_2\}$ admits the bipartition 
$A_2\cup B_2$. The statement about the number of edges is correct, too. 

Now suppose that $k\geq 3$ and that the statement is true for $k-1$. By ~\ref{lem:AndInd}, 
it is known that $\An_k - v_{3k-4}v_{3k-3}v_{3k-2} = \An_{k-1}$. By induction, $F_{k-1}$ is 
a minimal bipartification for $\An_{k-1}$ in the stronger, $5$-cycle-sense, and $A_{k-1}\cup B_{k-1}$ 
is a bipartition of $\An_{k-1}-F_{k-1}$. 

To prove the statement about being a bipartification for $k$, it suffices to show that in $\An_k$ every 
edge having at least one endvertex $v$ which is either \emph{new} 
(i.e. $v\in (A_k\cup B_k)\backslash (A_{k-1}\cup B_{k-1})$), 
or has \emph{changed partition classes} 
(i.e. $v\in A_{k}\cap B_{k-1}$ or $v\in B_{k} \cap A_{k-1}$), lies in $F_k$.

From the definition of $A_k$ and $B_k$ it is clear that $(A_k\cup B_k)\backslash (A_{k-1}\cup B_{k-1})=\{v_{3k-4}, v_{3k-3}, v_{3k-2}\}$, 
and that $v_{3k-4}\in B_k$, $v_{3k-3}\in A_k$ and $v_{3k-2}\in B_k$. 

As to $v_{3k-4}\in B_k$, from the definition of $\An_k$ it is clear that this vertex is adjacent to exactly 
the $k$ vertices in $\{v_{3k-3}\}\cup \{v_{3k-5-3j}:\; j=0,\dotsc, k-2\}$. Of these, exactly those in 
$\{v_{3k-5-3j}:\; j=0,\dotsc, \floor{\frac{k}{2}}-1\}$ lie in $B_k$. Therefore, it suffices to check that the edges in 
$\bigl \{\{v_{3k-4},v_{3k-5-3j}\}:\; j=0,\dotsc, \floor{\frac{k}{2}}-1\bigr \}$ lie in $F_k$. This becomes 
obvious by setting $i=0$ in \eqref{Uk1}.

As to $v_{3k-3}\in A_k$, from the definition of $\An_k$ it is clear that this vertex is adjacent to exactly 
the $k$ vertices in $\{v_{3k-2}\}\cup \{v_{2+3i}:\; i=0,\dotsc, k-2\}$, all of which lie in $B_k$.

As to $v_{3k-2}\in B_k$, from the definition of $\An_k$ it is clear that this vertex is adjacent to exactly the $k$ 
vertices $\{v_{3i}:\; i=0,\dotsc, k-1\}$, all of which lie in $A_k$.

From the definition of $A_k$ and $B_k$ it is clear by divisibility that 
$v\in B_{k} \cap A_{k-1}= \{v_{3i+1}:\; \floor{(k-1)/2},\dotsc, k-1 \} \cap \{v_{3i+1}:\; i=0,\dotsc, \floor{(k-2)/2}-1 \}$ 
and this intersection is clearly empty for every integer $k\geq 2$. Thus, a vertex in class $A$ never changes over to 
the class $B$.

Analogously, $v\in A_{k} \cap B_{k-1}= \{v_{3i+1}:\; i=0,\dotsc, \floor{(k-1)/2}-1 \}\cap \{v_{3i+1}:\; \floor{(k-2)/2},\dotsc, k-1 \}$, 
but now this is non-empty iff $k$ is odd with the intersection being equal to $\{v_{3\floor{\frac{k-2}{2}}+1}\}$ because the oddness of $k$ 
implies $\floor{(k-1)/2}-1 = \floor{(k-2)/2}$. From the definition of $\An_k$ it is easy to see that $v_{3\floor{\frac{k-2}{2}}+1}\in A_k$ 
is adjacent exactly to the $k$ vertices 
\begin{equation}\label{neighbourhood}
\{v_{3i}:\; i=0,\dotsc, \floor{(k-2)/2}\}\cup\{ v_{3i+2}:\; i=\floor{(k-2)/2},\dotsc, k-2\}.
\end{equation}
Of these, exactly those in the first set lie in $A_k$, so it remains to check that the set of edges 
\begin{equation}\label{edgeset}
\bigl \{ \{ v_{3\floor{\frac{k-2}{2}}+1},v_{3i}\}:\; i=0,\dotsc, \floor{(k-2)/2} \bigr \}
\end{equation}
is a subset of $F_k$. To see this, fix $j=\floor{\frac{k-1}{2}}-1-i$ in \eqref{Uk2}. Using  
$\floor{(k-1)/2}-1 = \floor{(k-2)/2}$, which implies 
$(3i+1)+3\bigl(\floor{\frac{k-1}{2}}-1-i\bigr)=3 \floor{\frac{k-2}{2}}+1$, 
it is clear that the subset of $U_k^{(2)}$ thus obtained is equal to \eqref{edgeset}. This completes 
the induction as far as being a bipartification is concerned.

For proving the strong minimality of $F_k$ using the strong minimality of $F_{k-1}$ 
(which is know by induction), it suffices to show that for every edge in 
$F_k\backslash F_{k-1}$, there is a $5$-cycle in $\An_k$ which intersects 
$F_k$ in this edge only and is disjoint from $F_{k-1}$, which implies 
that the edge in question is indispensable.

To prepare for the determination of the set $F_k\backslash F_{k-1}$, note that for every pair of integers $k_1\geq 2$ and 
$k_2\geq 2$, since none of the edges in $U_{k_1}^{(1)}$ contains a vertex with an index divisible by $3$ whereas 
every edge in $U_{k_2}^{(2)}$ does, the intersections $U_{k_1}^{(1)}\cap U_{k_2}^{(2)}$ and $U_{k_1}^{(2)}\cap U_{k_2}^{(1)}$ are 
both empty. In particular, for every integer $k\geq 2$, 
\begin{equation}
U_{k}^{(1)}\cap U_{k-1}^{(2)} = \varnothing\quad\textnormal{and}\quad U_{k}^{(2)}\cap U_{k-1}^{(1)} = \varnothing.
\end{equation}

Obviously, for sets $S_1,S_2,S_3,S_4$, the condition that $S_1\cap S_4=\varnothing$ and $S_2\cap S_3=\varnothing$ implies 
that $(S_1\cup S_2)\backslash (S_3\cup S_4) = (S_1\backslash S_3)\cup (S_2\backslash S_4)$, hence 
\begin{equation}\label{setdiff}
F_k\backslash F_{k-1} = (U_{k}^{(1)}\backslash U_{k-1}^{(1)})\cup (U_{k}^{(2)}\backslash U_{k-1}^{(2)})
\end{equation}

As to $U_{k}^{(1)}\backslash U_{k-1}^{(1)}$, let $e_k^{i,j}:= \{ v_{(3k-4)-3i}, v_{(3k-5)-3j}\}$ and 
$f^{i,j} := \{ v_{3i}, v_{(3i+1)+3j} \}$, note that $e_k^{i,j}=e_{k-1}^{i-1,j-1}$, and that, 
due to the different parities of the indices, for all integers $i_1,i_2,j_1,j_2\geq 2$, the two 
edges $e_k^{i_1,j_1}$ and $e_k^{i_2,j_2}$, and the two edges $f^{i_1,j_1}$ and $f^{i_2,j_2}$, 
are equal iff $i_1=i_2$ and $j_1=j_2$.

If $k$ is odd, then $\floor{\frac{k-2}{2}}=\floor{\frac{k-1}{2}}-1$, hence, by the 
criterion for equality of two edges  $f^{i_1,j_1}$ and $f^{i_2,j_2}$,
\begin{align}
U_{k}^{(2)}\backslash U_{k-1}^{(2)} & = \bigcup_{i=0}^{\floor{\frac{k-1}{2}}-1}\bigcup_{j=0}^{\floor{\frac{k-1}{2}}-1-i} \{ f^{i,j} \}
\backslash  \bigcup_{i=0}^{\floor{\frac{k-1}{2}}-2}\bigcup_{j=0}^{\floor{\frac{k-1}{2}}-2-i} \{ f^{i,j} \} \\ 
& = \bigcup_{i=0}^{\floor{\frac{k-1}{2}}-1} \{ f^{i,\floor{\frac{k-1}{2}}-1-i}\} = \bigcup_{i=0}^{\floor{\frac{k-1}{2}}-1} 
\bigl \{ \{ v_{3i}, v_{3\floor{\frac{k-1}{2}}-2} \}  \bigr\}
\end{align}
and $\floor{\frac{k-1}{2}}=\floor{\frac{k}{2}}$, hence
\begin{align}
U_{k}^{(1)}\backslash U_{k-1}^{(1)} & = \bigcup_{i=0}^{\floor{\frac{k}{2}}-1} \bigcup_{j=i}^{\floor{\frac{k}{2}}-1} \{ e_k^{i,j} \} 
\quad\backslash \bigcup_{i=0}^{\floor{\frac{k-1}{2}}-1} \bigcup_{j=i}^{\floor{\frac{k-1}{2}}-1} \{ e_{k-1}^{i,j} \} \\
& = \bigcup_{i=0}^{\floor{\frac{k}{2}}-1} \bigcup_{j=i}^{\floor{\frac{k}{2}}-1} \{ e_k^{i,j} \}
\quad\backslash \bigcup_{i=0}^{\floor{\frac{k}{2}}-1} \bigcup_{j=i}^{\floor{\frac{k}{2}}-1} \{ e_{k-1}^{i,j} \} \\
& = \bigcup_{i=0}^{\floor{\frac{k}{2}}-1} \bigcup_{j=i}^{\floor{\frac{k}{2}}-1} \{ e_{k-1}^{i-1,j-1} \}
\quad\backslash \bigcup_{i=0}^{\floor{\frac{k}{2}}-1} \bigcup_{j=i}^{\floor{\frac{k}{2}}-1} \{ e_{k-1}^{i,j} \} \\
& = \bigcup_{j=0}^{\floor{\frac{k}{2}}-1} \{ e_{k-1}^{-1,j-1} \} = \bigcup_{j=0}^{\floor{\frac{k}{2}}-1} \bigl \{ \{v_{3k-4}, v_{(3k-5)-3j}\} \bigr \}
\end{align}
where the penultimate equality is true by the criterion for equality of two edges $e_k^{i_1,j_1}$ and $e_k^{i_2,j_2}$. 
Using \eqref{setdiff} it follows that, if $k$ is odd, 
\begin{equation}\label{diffwhenkodd}
F_k\backslash F_{k-1} = \bigcup_{j=0}^{\floor{\frac{k}{2}}-1} \bigl\{ \{v_{3k-4}, v_{(3k-5)-3j}\} \bigr \}
\cup \bigcup_{i=0}^{\floor{\frac{k-1}{2}}-1} \bigl\{ \{ v_{3i}, v_{3\floor{\frac{k-1}{2}}-2} \} \bigr\}.
\end{equation}

If $k$ is even, then $\floor{\frac{k-2}{2}}=\floor{\frac{k-1}{2}}$, hence $U_{k}^{(2)}\backslash U_{k-1}^{(2)}=\varnothing$, and 
$\floor{\frac{k-1}{2}}=\floor{\frac{k}{2}}-1$, hence
\begin{align}
U_{k}^{(1)}\backslash U_{k-1}^{(1)} & = \bigcup_{i=0}^{\floor{\frac{k}{2}}-1} \bigcup_{j=i}^{\floor{\frac{k}{2}}-1} \{ e_{k-1}^{i-1,j-1} \}
\quad\backslash \bigcup_{i=0}^{\floor{\frac{k}{2}}-2} \bigcup_{j=i}^{\floor{\frac{k}{2}}-2} \{ e_{k-1}^{i,j} \} \\
& = \bigcup_{j=0}^{\floor{\frac{k}{2}}-1} \{ e_{k-1}^{-1,j-1} \} = \bigcup_{j=0}^{\floor{\frac{k}{2}}-1} \bigl \{ \{v_{3k-4}, v_{(3k-5)-3j}\} \bigr \},
\end{align}
showing that $U_{k}^{(1)}\backslash U_{k-1}^{(1)}$ is given by the same formula regardless of the parity of $k$. Using \eqref{setdiff}, 
it follows that, if $k$ is even, 
\begin{equation}\label{diffwhenkeven}
F_k\backslash F_{k-1} = \bigcup_{j=0}^{\floor{\frac{k}{2}}-1} \bigl\{ \{v_{3k-4}, v_{(3k-5)-3j}\} \bigr \}.
\end{equation}

To prove the indispensability of each of the edges in 
$\bigcup_{j=0}^{\floor{\frac{k}{2}}-1} \bigl\{ \{v_{3k-4}, v_{(3k-5)-3j}\} \bigr \}$, 
for every $k$ and every $j\in \{0,1,\dotsc, \lfloor k/ 2\rfloor - 1\}$ define
\begin{equation}
C_k^{(j)} := v_{3k-4} v_{(3k-5)-3j} v_{3\floor{\frac{k-1}{2}}} v_{3k-2} v_{3k-3} v_{3k-4},
\end{equation}
and note that this is a $5$-cycle since four of the needed five adjacencies are 
obvious and since the identity $k = \lfloor (k-1)/ 2\rfloor + \lfloor k/ 2\rfloor + 1$ 
implies that $(3k-5)-3j \geq 3\lfloor (k-1)/2 \rfloor$ for every integer $k$ and every 
$j\in \{0,1,\dotsc, \lfloor k/ 2\rfloor - 1\}$, so 
$\bigl |3\lfloor (k-1)/2 \rfloor - ((3k-5)-3j)\bigr | = (3k-5)-3j - 3\lfloor (k-1)/2 \rfloor
= 3(k-\lfloor (k-1)/2 \rfloor - j - 2)+1$. Since the expression in the parentheses is nonnegative 
for every $j\in \{0,1,\dotsc, \lfloor k/ 2\rfloor - 1\}$, this shows that the absolute value of the 
difference of the indices of the two vertices $v_{(3k-5)-3j}$ and $v_{3\lfloor (k-1)/2 \rfloor}$ is 
congruent to $1$ modulo 3, hence the vertices are adjacent.

To see that for every edge in $\bigcup_{j=0}^{\floor{\frac{k}{2}}-1} \bigl\{ \{v_{3k-4}, v_{(3k-5)-3j}\} \bigr \}$, 
the cycle $C_k^{(j)}$ intersects $F_k = U_k^{(1)}\cup U_k^{(2)}$ precisely in the edge $\{v_{3k-4}, v_{(3k-5)-3j}\}$, 
take the indices in the edge sets $U_k^{(1)}$, $U_k^{(2)}$ and $E(C_k^{(j)})$ modulo $3$. 
For $E(C_k^{(j)})$ this results in the `signature' $(\{2,1\},\{1,0\},\{0,1\},\{1,0\},\{0,2\})$, whereas every 
element of $U_k^{(1)}$ has signature $(\{2,1\})$ and every element of $U_k^{(2)}$ has $(\{0,1\})$. This shows that 
$C_k^{(j)}$ can intersect $U_k^{(1)}$ in at most the edge $\{v_{3k-4}, v_{(3k-5)-3j}\}$, which it does 
(when $i=0$ in  $U_k^{(1)}$), and $U_k^{(2)}$ in at most the three edges $\{v_{(3k-5)-3j}, v_{3\floor{\frac{k-1}{2}}}\}$, 
$\{v_{3\floor{\frac{k-1}{2}}}, v_{3k-2}\}$ and $\{v_{3k-2}, v_{3k-3}\}$, which it does not, since in $U_k^{(2)}$ the index 
which is divisible by three rises only as high as $3\lfloor (k-1)/2 \rfloor - 3$, which prevents an equality with 
any of the three edges. 

Moreover, to see that for every $j\in \{0,1,\dotsc, \lfloor k/2 \rfloor - 1\}$, 
the cycle $C_k^{(j)}$ is disjoint from $F_{k-1}$, repeat this argument but note that now \emph{all} of the 
cadidate-edges arising from considering the indices modulo $3$ fail to be actually contained 
in the intersection $E(C_k^{(j)})\cap F_{k-1}$ for reasons of magnitude of indices.

Since by \eqref{diffwhenkeven}, 
$F_k\backslash F_{k-1} = \bigcup_{j=0}^{\floor{\frac{k}{2}}-1} \bigl\{ \{v_{3k-4}, v_{(3k-5)-3j}\} \bigr \}$, 
this completes the induction and proves the minimality of $F_k$ in the case of even $k$. In the case of 
odd $k$, by \eqref{diffwhenkodd} one still has to prove the indispensability each of the edges 
$\bigcup_{i=0}^{\floor{\frac{k-1}{2}}-1} \{ v_{3i}, v_{3\floor{\frac{k-1}{2}}-2} \}$. To this end, 
for every $i\in \{0,\dotsc, \lfloor (k-1)/2 \rfloor\}$ define 
\begin{equation} 
D_k^{(i)} :=  v_{3i} v_{3\floor{\frac{k-1}{2}}-2} v_{3k-4} v_{3k-3} v_{3k-2} v_{3i} 
\end{equation} 
and note that this is a $5$-cycle since all of the five needed adjacencies are obvious (in the sense 
that it is obvious how to compute each of the absolute values of the difference of indices). 

Too see that for every edge in $\bigcup_{i=0}^{\floor{\frac{k-1}{2}}-1}\bigl\{ \{ v_{3i}, v_{3\floor{\frac{k-1}{2}}-2} \}\bigr\}$, 
the cycle $D_k^{(i)}$ intersects $F_k = U_k^{(1)}\cup U_k^{(2)}$ precisely in the edge 
$\{ v_{3i}, v_{3\floor{\frac{k-1}{2}}-2} \}$, repeat the argument given three paragraphs earlier. 

This time, the cycle has signature $(\{0,1\},\{1,2\},\{2,0\},\{0,1\}, \{1,0\})$. 
This shows that $D_k^{(i)}$ can intersect $U_k^{(1)}$ in at most the edge 
$\{v_{3\floor{\frac{k-1}{2}}-2}, v_{3k-4}\}$. Again by looking at the remainders modulo $3$ it 
is clear that for this to happen it is necessary that $i=0$ in $U_k^{(1)}$, whereupon the $j$ would 
have to satisfy $3\lfloor (k-1)/2 \rfloor - 2 = 3k - 5 -3j$ which is equivalent to $j=\lfloor k/2 \rfloor$ 
which contradicts $j\in \{0,\dotsc, \lfloor k/2 \rfloor -1 \}$, hence $D_k^{(i)}$ does not 
intersect $U_k^{(1)}$. Furthermore, the signatures of the cycle show that $D_k^{(i)}$ can 
intersect $U_k^{(2)}$ in at most the three edges 
$\{v_{3i}, v_{3\floor{\frac{k-1}{2}}-2}\}$, $\{v_{3k-3}, v_{3k-2}\}$ and $\{v_{3k-2}, v_{3i}\}$. Setting 
$j=\lfloor (k-1)/2 \rfloor - 1 - i$ in $U_k^{(2)}$ shows that the first of these edges actually lies 
in the intersection, and considering the magnitude of the index which is divisible by three in the second 
edge shows that the second edge does not. As to the third candidate-edge, since $j\leq \lfloor (k-1)/2 \rfloor - 1 - i$, 
hence $i+j \leq \lfloor (k-1)/2 \rfloor - 1$, implies that $(3i+1)+3j = 1 + 3(i+j)\leq  3\lfloor (k-1)/2 \rfloor - 2$, 
the index of the vertex which is not divisible by three cannot reach $3k-2$, so the edge is not in the intersection.

Moreover, to see that for every $i\in \{0,\dotsc, \lfloor (k-1)/2 \rfloor\}$ the cycle $D_k^{(i)}$ is disjoint 
from $F_{k-1}$, repeat the argument from above and note that all arguments for an edge not 
lying in the intersection $E(D_k^{(i)})\cap F_k$ can be adapted to the intersection $E(D_k^{(i)})\cap F_{k-1}$ 
and that the edge $\{v_{3i}, v_{3\floor{\frac{k-1}{2}}-2}\}$, which was the only one to make it into 
the intersection before, does not lie in $U_{k-1}^{(2)}$ since an analogous estimate as the one above now 
shows that $(3i+1)+3j\leq 3\lfloor (k-2)/2 \rfloor - 2$ and 
$3\lfloor (k-2)/2 \rfloor - 2 < 3\lfloor (k-1)/2 \rfloor - 2$ since $k$ is odd.

The statement about the cardinality of $F_k$ needs no induction. It is obvious that 
$| F_k\bigr | = \frac{1}{2}\floor{\frac{k}{2}} ( \floor{\frac{k}{2}} + 1 ) + \frac{1}{2}\floor{\frac{k-1}{2}}
( \floor{\frac{k-1}{2}} + 1 )$, and by distinguishing between odd and even $k$ it is easy to see that 
this is equal to $\floor{\frac{k^2}{4}}=\floor{\frac{1}{36}\bigl( |\An_k|+1\bigr)^2}$. 
\end{proof}
Since it is equally easy to show that $\floor{\frac{1}{36}\bigl( |\An_k|+1\bigr)^2}\leq \frac{1}{25}|\An_k|^2$, 
with the inequality being strict for every $k\geq 5$, Theorem \ref{thm:bipForAnd} is consistent with the 
following well-known conjecture of Paul Erd\H{o}s.
\begin{conjecture}[Erd\H{o}s bipartification conjecture; for more information see the introductions of 
\cite{BraLocal1998} and \cite{SudK4} and the references therein] 
Every triangle-free graph $G$ can be made bipartite by deleting at most $\frac{1}{25}|G|^2$ edges.
\end{conjecture}

\section{Concluding remarks}
There are two interesting questions concerning Theorem \ref{thm:bipForAnd}. 
\subsection{Is the bipartifcation $F_k$ minimum?} It is easy to see that there exist 
minimal bipartifications of $\An_k$ in the stronger sense of Theorem \ref{thm:bipForAnd}, 
which nevertheless have almost twice as many edges. An example is the set of all 
edges running between the vertex set $\bigl\{v_{2+3i}:\; i\in \{0,\dotsc, k-2\}\bigr\}$ 
and the vertex set $\bigl\{v_{3i}:\; i\in \{0,\dotsc, k-1\}\bigr\}$. This set consists of 
exactly $\frac12 (k-1) k $ edges, and it is easy to check that 
$\frac{|F_k|}{\frac12 (k-1) k}$ is strictly less than $\frac12 + \frac{k+1}{(k-1)k}$. Given 
so much variation in the cardinalities of minimal bipartifications, it is natural to 
wonder whether the minimal bipartification $F_k$ from Theorem \ref{thm:bipForAnd} is also \emph{minimum}, 
i.e. has the smallest cardinality a bipartification of $\An_k$ can have. The author thinks it likely 
that this is the case.
\begin{conjecture} 
The set $F_k$ defined in Theorem \ref{thm:bipForAnd} is a minimum bipartification for $\An_k$.
\end{conjecture}
However, $F_k$ is not unique in the sense that there are several other minimal bipartifications 
which are mutually edge-disjoint and have the same cardinality as $F_k$, and still more such bipartifications 
when one does not require mutual edge-disjointness. This leads to a second question concerning Theorem \ref{thm:bipForAnd}.
\subsection{Can Theorem \ref{thm:bipForAnd} be helpful for proving a somewhat less special special case of the Erd\H{o}s bipartification conjecture?}
The graphs $\An_k$ are important for characterizing triangle-free graphs with large minimum degree as homomorphic 
preimages. By Theorem 3.8 in ~\cite{ChenJinKoh1997}, later given a simpler proof in ~\cite{Bra1999}, if a 
triangle-free graph $G$ has minimum degree $\delta(G)>\frac13 |G|$ and chromatic number $\chi(G)\leq 3$, 
then it is homorphic to an Andr\'{a}sfai graph. This is why proving the following conjecture would be a 
little more than merely a drop in the ocean with regard to the Erd\H{o}s bipartification conjecture.
\begin{conjecture}\label{conj:hompre} By considering several copies of the  bipartification $F_k$, 
and then optimizing a system of quadratic inequalities, it is possible to prove that an 
arbitrary homomorphic preimage $H$ of $\An_k$ can be made bipartite by deleting at most 
$\frac{1}{25} |H|^2$ edges. 
\end{conjecture}
Furthermore, by Corollary 4.2 in ~\cite{BraTho4col}, every triangle-free graph $G$ with $\delta(G)>\frac13 |G|$ 
is four-colourable, and by Theorem 1.2, Corollary 1.3 and Theorem 1.4 in ~\cite{Bra1999}, if the chromatic 
number \emph{is} indeed four, then such graphs must simultaneously contain a Petersen graph with 
one contracted edge, a Wagner graph (M\"{o}bius ladder with four rungs) and a Gr\"{o}tzsch graph 
as subgraphs.

Therefore, proving Conjecture ~\ref{conj:hompre} (and thereby settling the 
case of $\delta(G) > \frac13 |G|$ and $\chi(G)\leq 3$) would allow anyone 
interested in the Erd\H{o}s bipartification conjecture to assume that one of the following holds:
\begin{enumerate}
\item The minimum degree of the graph $G$ is at most $\frac13 |G|$. If $\delta(G)<\frac13 |G|$, 
then the chromatic number $G$ can be arbitrarily high (see Section 6 in  ~\cite{BraTho4col}).
\item The minimum degree of the graph $G$ is strictly larger than $\frac13 |G|$,  the chromatic 
number of $G$ is exactly four, and the graph $G$ contains a Petersen graph with one edge contracted, 
a Wagner graph, and a Gr\"{o}tzsch graph as subgraphs.
\end{enumerate}

\section*{Acknowledgements} The author is very grateful to Anusch Taraz for introducing 
him to the Erd\H{o}s bipartification conjecture and to both the TopMath program 
and the Department M9 for Applied Geometry and Discrete Mathematics of Technische 
Universit\"{a}t M\"{u}nchen for support and excellent working conditions.


\bibliographystyle{amsplain} \bibliography{andrasfaibipartification20090929}


\end{document}